\documentclass{amsart}
\usepackage{amsfonts}
\usepackage{latexsym}
\usepackage{hyperref}
\usepackage{amssymb}
\usepackage{amsmath}
\usepackage{enumerate}
\usepackage{latexsym,color,amsmath,amsthm,amssymb,amscd,amsfonts}
\usepackage{graphicx}
\usepackage[normalem]{ulem}

\usepackage[utf8]{inputenc}
\usepackage[T1]{fontenc}

\newtheorem{thm}{Theorem}[section]
\newtheorem{cor}[thm]{Corollary}
\newtheorem{lemma}[thm]{Lemma}

\theoremstyle{remark}

\newtheorem{problem}{Problem}

\begin{document}


\title[Sobolev inequalities for log-concave functions]{Sobolev inequalities for sections and projections of log-concave functions}

\author[L. J. Al\'ias]{Luis J. Al\'ias}
\email{ljalias@um.es}
\address{Department of Mathematics. Universidad de Murcia, Spain}

\author[B. Gonz\'alez Merino]{Bernardo Gonz\'alez Merino}
\email{bgmerino@um.es}
\address{Department of Engineering and Technology of Computers, area of Applied Mathematics.
Universidad de Murcia, 
Spain}

\author[B. Mar\'in Gimeno]{Beatriz Mar\'in Gimeno}
\email{b.maringimeno@um.es}
\address{Department of Mathematics. Universidad de Murcia, Spain}


\thanks{This research has been partially supported by the grants  PID2025-167696NB-I00, and PID2025-169631NB-I00, both of them funded by MICIN/AEI, Spain. The third author is also supported by Contratos
Predoctorales FPU-Universidad de Murcia 2024, Spain.}

\date{\today}

\begin{abstract}
Higher order sharp Sobolev and local Sobolev type inequalities for log-concave functions are established. The new inequalities directly imply the isoperimetric inequalities by Bollob\'as and Thomason in 1995, by Liakopoulos in 2019, and the recently discovered local isoperimetric inequalities from \cite{AGM26}. In order to do so, we make use of three ingredients. First, sets of functional inequalities proven by Alonso-Gutiérrez, Bernués, Brazitikos and Carbery in 2021. Second, a multidimensional polynomial optimization result that originated in \cite{AGM26} and was initiated by Bollob\'as and Thomason. Third, a new integral formula for the projection of a log-concave function that links the geometric part with the more functional versions via some dependence on the Hessian matrix restricted to subspaces.
\end{abstract}

\maketitle

\section{Introduction}

Sobolev inequality lies in the interface between analysis and geometry \cite{O78,T76}. This is a reformulation of the classical isoperimetric inequality. It states that for every $C^1(\mathbb R^n)$ function $f$ with compact support there exists a characteristic function $g$ of a Euclidean ball such that 
\[
\int_{\mathbb R^n}f(x)^\frac{n}{n-1}dx = \int_{\mathbb R^n} g(x)^{\frac{n}{n-1}}dx
\quad\text{and}\quad 
\int_{\mathbb R^n}\|\nabla f(x)\|_2dx \geq \int_{\mathbb R^n} \|\nabla g(x)\|_2dx.
\]
Above $\nabla f$ denotes the gradient of $f$, and for every $x\in\mathbb R^n$, $\|x\|_2:=\sqrt{\sum_{i=1}^n x_i^2}$ denotes the Euclidean norm of $x$.  

Log-concave functions play a crucial role in asymptotic geometric analysis and convex geometry \cite{KM05}. A function $f:\mathbb R^n\rightarrow[0,\infty)$ is a log-concave function if $\log f$ is a concave function, i.e., if $f((1-\lambda)x+\lambda y)\geq f(x)^{1-\lambda}f(y)^\lambda$ for every $x,y\in\mathbb R^n$ and $\lambda\in[0,1]$. If a log-concave function has finite integral, we say it is integrable. We call $\mathcal F(\mathbb R^n)$ the set of log-concave integrable functions whose domain is $\mathbb R^n$. Log-concave functions naturally extend the space of convex bodies $\mathcal K^n$, i.e.~convex and compact sets with non-empty interior of $\mathbb R^n$, either by means of characteristic functions or by exponential functions of the form
\[
\begin{array}{cccc}
    f: & \mathbb R^n & \longrightarrow & [0,\infty)  \\
     & x & \rightsquigarrow & f(x):=e^{-\|x\|_K},
\end{array}
\]
where $0\in K\in\mathcal K^n$ and $\|x\|_K:=\inf\{\rho\geq 0:x\in\rho K\}$ is the Minkowski functional of $K$. Moreover, we say that $K\subseteq\mathbb R^n$ is strictly convex if $(1-\lambda)x+\lambda y\in\mathrm{int}(K)$ for every $x,y\in K$ such that $x\neq y$ and where $\mathrm{int}(K)$ means the interior of $K$.

Several efforts have been made to extend convex geometry to the space of log-concave functions (see \cite{AGM25, AAGJV,AGV23, B88,C06,LSW21,L22}). These efforts intertwined with several Sobolev type results (see \cite{A79,B10,Bu26,CLM25,CENV04,Ga58,H09,1HLPRY25,2HLPRY25,HL25,L86,LYZ00, Ni59,X07,Z91,Z99}).

For any $\sigma\subseteq[n]:=\{1,2,\dots,n\}$, we denote by $H_{\sigma}:=\langle\{e_i:i\in\sigma\} \rangle$ be the linear hull of $\{e_i:i\in\sigma\}$, where $\{e_1,\dots,e_n\}$ is the canonical basis of $\mathbb R^n$. Log-concave functions can naturally be projected onto linear subspaces. If $f$ is a log-concave function within $\mathbb R^n$ and $H$ is a linear subspace, let
\[
P_Hf(x):=\sup\left\{f(x+y):y\in H^\bot\right\}
\]
be the orthogonal projection of $f$ onto $H$.
If $f\in\mathcal F(\mathbb R^n)$ and $H$ is a linear subspace, then $P_Hf\in\mathcal F(H)$ (see \cite[Pg.~178]{KM05} for further details). This notion generalizes the usual orthogonal projections $P_H(K)$ of any $K\in\mathcal K^n$ onto any given subspace $H$. In fact, $P_H e^{-\|\cdot\|_K} = e^{-\|\cdot\|_{P_HK}}$ for every $K\in\mathcal K^n$ with $0\in K$ and $H$ any linear subspace. Moreover, for a  function $f:\mathbb R^n\to\mathbb R^m$ such that $f\in\mathcal{C}^1(\mathbb R^n,\mathbb R^m)$ we denote by $J_{H_\tau}f$ the Jacobian matrix of $f=(f_1,\dots,f_m)$ computed in $H_\tau$ for some $\tau \subseteq[n]$, i.e.,
$$\left(J_{H_\tau}f\right)_{i,j}:=\partial_jf_i$$ for every $j\in\tau$ and $i\in[m]$. In particular, if $m=1$, it is the gradient of $f$ computed in $H_\tau$ and it is denoted by 
$\nabla_{H_\tau}f$.
What is more, if $f:\mathbb R^n\to\mathbb R$ is a function such that $f\in\mathcal{C}^2(\mathbb R^n)$, the Hessian matrix computed with respect to $H_\tau$ is denoted by $D^2_{H_\tau}f$, i.e.,
\[
(D^2_{H_\tau}f)_{i,j}:=\partial^2_{ij}f
\]
for every $i,j\in\tau$. For $\tau\subseteq\sigma\subseteq[n]$, we define the operator
\[
    M(f,\sigma,\tau)(x):= \frac{\mathrm{det}\left(\left.-D^2_{H_{([n]\setminus\sigma)\cup\tau}^\bot}f(x)\right|_{\nabla_{H_{([n]\setminus\sigma)\cup\tau}^\bot}f(x)^\bot\cap H_{([n]\setminus\sigma)\cup\tau}^\bot}\right)}{ \left\Vert \nabla_{H_{([n]\setminus\sigma)\cup\tau}^\bot}f(x)\right\Vert_2^{|\sigma|-|\tau|-2}},
    \]
    where $f\in\mathcal{C}^2(\mathbb R^n)$ and such that $\left\Vert \nabla_{H_{([n]\setminus\sigma)\cup\tau}^\bot}f(x)\right\Vert_2\neq 0.$ Notice that  $D^2F(x)|_{\nabla F(x)^\bot}$ is the Hessian matrix of $F$ restricted to $\nabla F(x)^\bot$ for some $F\in\mathcal{C}^2(\mathbb R^n)$, i.e. if $\langle u_i:i=1,\ldots,k\rangle$ is an orthonormal basis of $\nabla F(x)^\bot,$ then
 \[
 D^2F(x)|_{\nabla F(x)^\bot}=U^tD^2F(x)U
 \] 
 where $U=\left(u_1,\ldots,u_k\right)$.

Bollob\'as and Thomason \cite{BT95} introduced the notion of covers. Let $\sigma\subseteq[n]$ and let $s\geq 1$. We say that 
the tuple $(\sigma_1,\dots,\sigma_m)$ is a $s$-cover of $\sigma$ if $\sigma_i\subseteq\sigma$, $i=1,\dots,m$, and $|\{j:i\in\sigma_j\}|=s$ for every $i\in\sigma$.

Alonso-Gutiérrez, Bernués, Brazitikos and Carbery \cite[Thm. 4.2]{ABBC21} 
proved a functional version of the classical Bollobás-Thomason inequality \cite[Thm.~2]{BT95}
\begin{equation}\label{eq: functional ineq proyecciones BT}
    \Vert f\Vert_\infty^{m-s}\left(\int_{\mathbb{R}^n}f(x)dx\right)^s\leq\frac{(n!)^s}{\prod_{j=1}^m|\sigma_j|!}\prod_{j=1}^m\int_{H_{\sigma_j}}P_{H_{\sigma_j}}f(x)dx,
\end{equation}
where $f\in\mathcal{F}(\mathbb{R}^n)$ and $(\sigma_1,\ldots,\sigma_m)$ is a $s$-cover of $[n].$ Moreover, if $f=e^{-\|\cdot\|_B}$, where $B$ is a coordinate box centered at the origin (i.e. a convex body of the form $B=\prod_{i=1}^n[-a_i,a_i])$, equality holds above.

 Our first result is a Sobolev type inequality where the extremizer is a log-concave function of the form $e^{-\|\cdot\|_B}$ for some box $B$ depending on $f$. To do so, let us introduce the class of functions
\[
\mathcal{Q}K(\mathbb R^n):=\{f:\mathbb R^n\to[0,\infty): L_{f,t} \text{ is convex and }\kappa_{\text{bd}(L_{f,t})}>0,\, \forall t \in (0,\Vert f\Vert_\infty)\},
\]
where $L_{f,t}:=\{x\in\mathbb R^n:f(x)\geq t\}$ for every $t\in(0,\Vert f\Vert_\infty)$ and $\kappa_{\text{bd}(L_{f,t})}$ is the Gauss-Kronecker curvature of $\text{bd}(L_{f,t})$. 

\begin{thm}\label{thm:SobolevBTineq}
    Let $f\in\mathcal{F}(\mathbb{R}^n)\cap \mathcal{C}^2(\mathbb{R}^n)\cap \mathcal{Q}K(\mathbb R^n)$ be such that $\Vert f\Vert_\infty<\infty$ and, for almost every $x\in\mathrm{supp}(f)$, $\nabla_{H_{\{j\}}} f(x)\neq 0$ for every $j\in[n]$. Then, there exists a coordinate box $B_f$ such that 
    \[
    \int_{\mathbb{R}^n}\frac{f(x)}{\Vert f\Vert_\infty}dx=\int_{\mathbb{R}^n}e^{-\Vert x\Vert_{B_f}}dx
    \]
    and
    \[
\frac{1}{\|f\|_\infty}\int_{\mathbb R^n} M(f,[n],\tau)(x) dx \geq \int_{\mathbb R^n} M\left(e^{-\|\cdot\|_{B_f}},[n],\tau\right)(x) dx
\]
    for any $\tau\subset[n].$
\end{thm}

The right-hand side of the second inequality in Theorem \ref{thm:SobolevBTineq} is understood as the value of
$M(e^{-\|\cdot\|_{B_f}},[n],\tau)$ obtained as the limit of smooth approximations, i.e. $g_k\in\mathcal{F}(\mathbb{R}^n)\cap \mathcal{C}^2(\mathbb{R}^n)\cap \mathcal{Q}K(\mathbb R^n)$ such that
\[
g_k \underset{k\to\infty}{\longrightarrow} e^{-\|\cdot\|_{B_f}} \quad\text{and}\quad\int_{H_\tau}P_{H_\tau }g_k(x)dx\underset{k\to\infty}{\longrightarrow}\int_{H_\tau}P_{H_\tau} e^{-\|x\|_{B_f}}dx.
\]
Moreover, the limit value coincides with
\[
|\tau|!S(\mathbb B^{n-|\tau|})\mathrm{vol}_{|\tau|}\left(P_{H_\tau}B_f\right),
\]
where $S(\mathbb B^{n-|\tau|})$ denotes the surface area of  $\mathbb B^{n-|\tau|}$, and $\mathrm{vol}_{|\tau|}\left(P_{H_\tau}B_f\right)$ denotes the $|\tau|$-dimensional volume of $P_{H_\tau}B_f$.
Let us observe that when $f=e^{-\|\cdot\|_K}$ for some $K\in\mathcal K^n$ with $0\in K$, then Theorem \ref{thm:SobolevBTineq} becomes the isoperimetric type result \cite[Thm.~1]{BT95}. 

Again the same authors proved (see \cite[Thm. 5.1]{ABBC21}) a functional dual to \eqref{eq: functional ineq proyecciones BT} which generalizes Liakopoulos result \cite{L19} (see also \cite[Thm.~1.4]{L19}). They showed that 
\begin{equation}\label{eq: ABBC funcional secciones}
    \left(n^n\int_{\mathbb{R}^n}f^n(x)dx\right)^s\geq\prod_{j=1}^m|\sigma_j|^{|\sigma_j|}\int_{H_{\sigma_j}}f^{|\sigma_j|}(x)dx,
\end{equation}
where $f\in\mathcal{F}(\mathbb{R}^n)$ and $(\sigma_1,\ldots,\sigma_m)$ is a $s$-cover of $[n].$ Moreover, if $f=e^{-\|\cdot\|_C}$, where $C$ is a crosspolytope centered at the origin (i.e. a convex body of the form $C=\mathrm{conv}(\{\pm\lambda_ie_i:\,i=1,\ldots,n\})$ where $\mathrm{conv}(\cdot)$ means the convex hull), equality holds above. 

Our second result is a Sobolev type result where the extremizer is a function of the form $e^{-\|\cdot\|_C}$ with $C$ a certain crosspolytope. 
\begin{thm}\label{thm:SobolevLiakopoulos}
    Let $f\in\mathcal{F}(\mathbb{R}^n)$ such that $0\in\mathrm{int}(\mathrm{supp}(f))$. Then, there exists $C_f$  a crosspolytope, depending on $f$, such that 
    \[
\int_{\mathbb{R}^n}f^n(x)dx=\int_{\mathbb{R}^n}\left(e^{-\Vert x\Vert_{C_f}}\right)^ndx
    \]
    and
    \[
\int_{H_\tau}f^{|\tau|}(x)dx\leq\int_{H_\tau}\left(e^{-\Vert x\Vert_{C_f}}\right)^{|\tau|}dx
    \]
    for any $\tau\subseteq [n]$.
\end{thm}
Observe that when $f=e^{-\|\cdot\|_K}$ for some $K\in\mathcal K^n$ with $0\in K$, then Theorem \ref{thm:SobolevLiakopoulos} becomes the isoperimetric type result \cite[Thm.~1.3]{L19}.

The third set of inequalities, proven by the same authors in \cite[Thm.~4.3]{ABBC21}, are local versions of \eqref{eq: functional ineq proyecciones BT},
\begin{equation}\label{eq: funcional local BT}
\begin{split}
        \left(\int_{H_\sigma^\bot}P_{H_\sigma^\bot}f(x)dx\right)^{m-s} &  \left(\int_{\mathbb{R}^n}f(x)dx\right)^s \\ & \leq\frac{|\sigma|!^s}{\prod_{j=1}^m|\sigma_j|!}\prod_{j=1}^m\int_{H_\sigma^\bot\oplus H_{\sigma_j}}P_{H_\sigma^\bot\oplus H_{\sigma_j}}f(x)dx,
\end{split}
\end{equation}
where $f\in\mathcal F(\mathbb R^n)$, $\sigma\subset[n]$,  $(\sigma_1,\dots,\sigma_m)$ is a $s$-cover of $\sigma$, and $H_\sigma^\bot$ is the orthogonal subspace to $H_\sigma$.
These functional inequalities recover the geometric inequalities \cite[Thm.~1.5]{ABBC21} when considering $f=e^{-\Vert\cdot\Vert_K}$ for some $K\in\mathcal{K}^n$ containing the origin.

Our third result is a local Sobolev type result where the extremizer is a function of the form $e^{-\|\cdot\|_K}$ with $K$
a certain convex hull of a lower dimensional box and another lower dimensional set on the orthogonal complement to the subspace that contains the box.

\begin{thm}\label{thm:localSobolevBT}
    Let $\sigma\subset[n]$. Let $f\in\mathcal{F}(\mathbb{R}^n)\cap \mathcal{C}^2(\mathbb{R}^n)\cap \mathcal{Q}K(\mathbb R^n)$ be such that $\Vert f\Vert_\infty<\infty$, $\nabla_{H_{\{j\}}} f(x)\neq 0$, for every $j\in\sigma$, and for almost every $x\in \mathrm{supp}(f)$. Then, there exists a certain convex body
    \[K_f=\mathrm{conv}\left(\left\{\sum_{i\in\sigma}c_i[-e_i,e_i],L
    \right\}\right),\]
     where $c_i>0$, $i\in\sigma$ and $L$ is a $(n-|\sigma|)$-dimensional convex body contained in $H_\sigma^\bot$ with $0\in L$
    such that   
    \[
    \int_{\mathbb{R}^n}f(x)dx=\int_{\mathbb{R}^n}e^{-\Vert x\Vert_{K_f}}dx,
    \]
    \[
    \int_{\mathbb R^n}M(f,\sigma,\tau)(x)dx\geq \int_{\mathbb R^n}M\left(e^{-\Vert \cdot\Vert_{K_f}},\sigma,\tau\right)(x)dx
    \]
    for every $\tau\subset\sigma$ and
    \[
    \int_{\mathbb R^n}M(f,\sigma,\emptyset)(x)dx=\int_{\mathbb R^n}M\left(e^{-\Vert x\Vert_{K_f}},\sigma,\emptyset\right)(x)dx.
    \]
\end{thm}
 The right-hand side of the second inequality and third equality in Theorem \ref{thm:localSobolevBT} are understood as the value of
$M\left(e^{-\|\cdot\|_{K_f}},\sigma,\tau\right)$ obtained as the limit of smooth approximations like in Theorem \ref{thm:SobolevBTineq}.
Moreover, the limit value of the right-hand side of the second equation in Theorem \ref{thm:localSobolevBT} coincides with
\[(n-|\sigma|)!|\tau|!
S(\mathbb B^{|\sigma|-|\tau|})\mathrm{vol}_{n-|\sigma|}\left(L\right)2^{|\tau|}\prod_{i\in\tau}c_i
\]
and limit value of the right-hand side of the third equation is equal to
\[
S(\mathbb B^{|\sigma|})(n-|\sigma|)!\mathrm{vol}_{n-|\sigma|}(L).
\]
Let us observe that when $f=e^{-\|\cdot\|_K}$ for some $K\in\mathcal K^n$ with $0\in K$, then Theorem \ref{thm:localSobolevBT} becomes the local isoperimetric type result \cite[Thm.~1.1]{AGM26}. Notice also that the $(n-|\sigma|)$-dimensional convex body contained in $H_\sigma^\bot$, $L$, is not unique.

Notice that we actually prove slightly broader inequalities than the ones in Theorems \ref{thm:SobolevBTineq} and \ref{thm:localSobolevBT}. Their connection emanates from the new identity 
\begin{equation}\label{eq: diferential formula integral projection}
\int_{H_\sigma}P_{H_\sigma}f(y)dy=\frac{1}{S(\mathbb B^{n-|\sigma|})}\scalebox{1.5}{$\displaystyle\int$}_{\mathbb R^n}\frac{\mathrm{det}\left(\left.-D^2_{H_\sigma^\bot}f(x)\right|_{\nabla_{H_\sigma^\bot} f(x)^\bot\cap H_\sigma^\bot}\right)}{\Vert \nabla_{H_\sigma^\bot} f(x)\Vert_2^{n-|\sigma|-2}}dx,
  \end{equation}
where $\sigma\subseteq[n]$, $f\in\mathcal{C}^2(\mathbb{R}^n)\cap \mathcal{Q}K(\mathbb R^n)$ is an integrable function such that $\Vert f\Vert_\infty<\infty$ and, for almost every $x\in\mathrm{supp}(f)$, $\nabla_{H_\sigma^\bot} f(x)\neq 0$. 
In particular, for $f\in\mathcal{F}(\mathbb R^n)\cap\mathcal{ C}^1(\mathbb R^n)$ with compact support and $\sigma=[n]\setminus\{j\}$, for some $j\in[n]$, \eqref{eq: diferential formula integral projection} recovers the well known identity given by 
\begin{equation}\label{eq:FunctCauchy}
    2\int_{H_{\{j\}}^\bot}P_{H_{\{j\}}^\bot}f(x)dx = \|\nabla_{H_{\{j\}}}f\|_1,
\end{equation}
where $\nabla_{H_{\{j\}}}f$ denotes the partial derivative of $f$ in the direction $e_j$ (see \cite{AGM15, BCF14, Z99}). Therefore, it extends the functional version of the Cauchy surface area formula \eqref{eq:FunctCauchy} to projections onto subspaces of lower dimension. Many authors have worked with the integral of the projection of a log-concave function but from different points of view such as the functional version of a quermassintegral, see \cite{BCF14, CH00, CH05, CF13, CLM25, L22,Tn26}.


Bollob\'as and Thomason mentioned in their paper that \cite[Thm.~1]{BT95} is also true for general product measures. The proof in their work covers in detail the case of the Lebesgue measure. However, their proof strongly depends on the homogeneity of the Lebesgue measure, fact which is generally not true for general measures. For the sake of completeness and clearness, we provide here an exhaustive proof of that fact.

Product measures play a crucial role in asymptotic convex geometry. For instance,  processes such as Brownian motions, Markov chains in the independent case, Wiener measure, the Shannon entropy, or even machine learning require product measures to be built.

Let $\mu_j$ be a continuous measure over $\mathbb R$ (i.e. $\mu_j(\{x\})=0$ for every $x\in\mathbb R$) for every $j=1,\dots,n$. We then define the product measure $\underset{j\in[n]}{\otimes}\mu_j:=\mu_1\otimes\cdots\otimes\mu_n$ to be
\[
\mu_1\otimes\cdots\otimes\mu_n:\prod_{j=1}^n\mathrm{supp}(\mu_j)\rightarrow[0,\infty) \quad\text{s.t.}\quad \underset{j\in[n]}{\otimes}\mu_j \left(\prod_{j=1}^n L_j\right) :=\prod_{j=1}^n\mu_j(L_j),
\]
where $L_j\subset\mathbb R$ is a $\mu_j$-measurable set for every $j=1,\dots,n$.

Bollob\'as and Thomason pointed out that for every $K\in\mathcal K^n$ and $(\sigma_1,\dots,\sigma_m)$ be a $s$-cover of $[n]$, then
\begin{equation}\label{eq:BTprodMeas}
    \mu_1\otimes\cdots\otimes \mu_n(K)^s \leq \prod_{i=1}^m\underset{j\in\sigma_i}{\otimes} \mu_j(P_{H_{\sigma_i}}K).
\end{equation}
As a consequence of the same technique employed above, we can also derive \cite[Thm.~1]{BT95}.

\begin{thm}\label{thm:ProductMeasure}
Let $\mu_j$ be a continuous measure on $\mathbb R$, $j=1,\dots,n$, and $K\in\mathcal K^n$ such that $K\subseteq\mathrm{supp}(\mu_1\otimes\cdots\otimes\mu_n)$. Then there exists a coordinate box $B_K\subseteq\mathrm{supp}(\mu_1\otimes\cdots\otimes\mu_n)$ such that
    \[\mu_1\otimes\cdots\otimes\mu_n(K)=\mu_1\otimes\cdots\otimes\mu_n(B_K)\]
    and 
    \[\underset{j\in\tau}{\otimes}\mu_j(P_{H_\tau}K)\geq\underset{j\in\tau}{\otimes}\mu_j(P_{H_\tau}B_K)\]
    for any $\tau\subseteq[n].$    
\end{thm}
This enables remarkable cases of isoperimetric inequalities for product measures as the case of the Gaussian measure.

Let us introduce further notation. For every $X\subset\mathbb R^n$, let $\mathrm{conv}(X)$ be the convex hull of $X$. Moreover, for every $x,y\in\mathbb R^n$, let $[x,y]:=\mathrm{conv}(\{x,y\})$ be the line segment of endpoints $x$ and $y$. For every $K,C\in\mathcal K^n$, let $K+C:=\{x+y:x\in K,\,y\in C\}$ be the Minkowski addition of $K$ and $C$. For every $\lambda \in\mathbb R$, let $\lambda K:=\{\lambda x:x\in K\}$ and let $-K:=(-1)K$. 

The paper is organized as follows. In Section \ref{sec:AlgebraicMethod} we extend an idea developed in \cite{AGM26} to a wider setting. The idea moves around a multidimensional polynomial optimization technique. It plays a crucial role in proving all other results of the paper. In Section \ref{Sec: integral formula for PHf} we prove the integral formula given in \eqref{eq: diferential formula integral projection} together with some technical lemmas. In Section \ref{sec:SobolevInequs} we prove the Sobolev type inequalities explained in Theorems \ref{thm:SobolevBTineq}, \ref{thm:SobolevLiakopoulos}, and \ref{thm:localSobolevBT}. In Section \ref{sec:ProdMeasure} we prove Theorem \ref{thm:ProductMeasure} by using again the technique developed in Section \ref{sec:AlgebraicMethod}. Finally in Section \ref{sec:furtherRemarks} we leave some open questions on the topic.




\section{Multidimensional polynomial optimization tool}\label{sec:AlgebraicMethod}
This whole section is dedicated to deriving a tool that was originally implemented by Bollob\'as and Thomason \cite{BT95}, used in the dual case by Liakopoulos \cite{L19}, and recently extended in \cite{AGM26}. Here we provide a more general version, that can be used either for geometrical or analytical results.

Here we present some basic notions of covers. We say that $(\sigma_1,\dots,\sigma_m)$ is a cover of $\sigma\subseteq[n]$ if $(\sigma_1,\dots,\sigma_m)$ is a $s$-cover of $\sigma$ for some $s\geq 1$. We say that any $s$-cover $(\sigma_1,\dots,\sigma_m)$ of $\sigma\subseteq[n]$ is irreducible if $(\sigma_1,\dots,\sigma_m)$ cannot be separated into two subsets forming each of them a cover of $\sigma$ (see \cite{BT95}). It is known that the number of irreducible $s$-covers of $[n]$ is finite (see \cite{BT95}). 

In the following, we use the definitions below through all the section. Let us fix $\sigma\subseteq[n]$, and let $C_\tau \geq 0$ for every $\tau\subseteq\sigma$. Let us consider the set $\mathcal C_\sigma$ given by all tuples $(x_\tau)_{\tau\subseteq\sigma}$ fulfilling
\begin{equation}\label{eq:Algebraic}
\left\{\begin{array}{ll}
    x_\sigma = C_\sigma &  \\
    &\\
    x_{\sigma\setminus\tau} \leq C_{\sigma\setminus\tau}  & \\
    &\\
    x_\sigma^{m-s} \leq \prod_{j=1}^m x_{\sigma\setminus\sigma_j} & \forall (\sigma_1,\dots,\sigma_m)\text{ irreducible }s\text{-cover of }\sigma \\
    &\\
    x_\tau \leq \prod_{j\in\tau} x_{\{j\}}&\text{ for any $\tau\subseteq\sigma.$}
\end{array}\right.
\end{equation}
We call the set of inequalities on the third line as Type-$(1)$, and the set of inequalities on the fourth line as Type-$(2)$. Notice that due to the finiteness of irreducible $s$-covers of $\sigma\subseteq[n]$, the set of inequalities above is finite. 

Now we are considering the following order relation. Given $(x_{\tau})_{\tau\subseteq\sigma}, (y_{\tau})_{\tau\subseteq\sigma}\in\mathcal{C}_\sigma$, let us define the order 
\[
(x_\tau)_{\tau\subseteq\sigma}\prec (y_\tau)_{\tau\subseteq\sigma} \quad :\Leftrightarrow \quad x_\tau\leq y_\tau\text{  for all }\tau\subseteq\sigma.
\]
Then, $(\mathcal{C}_\sigma,\prec)$ is a (partially) ordered set. 

\begin{thm}\label{thm:AlgebraicTool}
    Let $\sigma\subseteq[n]$ and let $\mathcal C_\sigma$ be defined by \eqref{eq:Algebraic}. If $\mathcal C_\sigma\neq\emptyset$, then it follows:
    \begin{itemize}
        \item[(i)] There exist $C(\sigma),D(\sigma)>0$ such that for every $(x_\tau)_{\tau\subseteq\sigma}\in\mathcal C_\sigma$ we have that $C(\sigma) \leq x_\tau \leq D(\sigma)$ for every $\tau\subseteq\sigma$.
        \item[(ii)] There exists a tuple $(t_\tau)_{\tau\subseteq\sigma} \in \mathcal C_\sigma$ which is minimal for the order $\prec$. 
        \item[(iii)] For every $\tau\subseteq\sigma$ we have that $t_\tau = \prod_{j\in\tau}t_{\{j\}}$.
    \end{itemize}
\end{thm}

\begin{proof}
We start proving $(i)$. Let us notice that from the second inequality in \eqref{eq:Algebraic} we have that 
$x_\tau \leq C_\tau$, for every $\tau\subseteq\sigma$, and let $D(\sigma):=\max\{C_\tau:\tau\subseteq\sigma\}$. For a fixed $\tau\subseteq\sigma$, consider the irreducible $1$-cover $\{\sigma\setminus\tau\}\cup\{i:i\in\tau\}$ of $\sigma$. By the Type-$1$ inequality in \eqref{eq:Algebraic} we obtain that
\[
x_\sigma^{|\tau|}=x_\sigma^{|\tau|+1-1} \leq x_\tau\left(\prod_{i\in\tau}x_{\sigma\setminus\{i\}}\right),
\]
for any $(x_\tau)_{\tau\subseteq\sigma}\in\mathcal{C}_\sigma.$
Then
\[
x_\tau \geq \frac{x_\sigma^{|\tau|}}{\prod_{i\in\tau}x_{\sigma\setminus\{i\}}} \geq 
\frac{x_\sigma^{|\tau|}}{\prod_{i\in\tau}C_{\sigma\setminus\{i\}}} =: c_\tau.
\]
Finally, let $C(\sigma):=\min\{c_\tau:\tau\subseteq\sigma\}$. Then for every $(x_\tau)_{\tau\subseteq\sigma}\in\mathcal C_\sigma$ we have that $C(\sigma) \leq x_\tau \leq D(\sigma)$ for every $\tau\subseteq\sigma$. 

We now show $(ii)$. Let $(x^\ell_\tau)_{\tau\subseteq\sigma}\in\mathcal C_\sigma$ be a decreasing chain on $\ell$, i.e.~a decreasing sequence on $\ell$ with respect to $\prec$. We aim to show that we find an infimum of this sequence within $\mathcal C_\sigma$. Notice that $x^\ell_\tau$ is a real decreasing sequence on $\ell$ for every $\tau\subseteq\sigma$. Thus, let $x^0_\tau:=\inf x^\ell_{\tau}=\lim_{\ell\rightarrow\infty}x^\ell_\tau$. It remains to show that $(x_\tau^0)_{\tau\subseteq\sigma}\in\mathcal C_\sigma$. The fact that the sequence $x^0_\tau$, $\tau\subseteq\sigma$, fulfills first and second inequalities of \eqref{eq:Algebraic} is obvious. Let us consider $(\sigma_1,\dots,\sigma_m)$ be a $s$-cover of $\sigma$. For every $\ell\in\mathbb N$ we have that
\[
(x^\ell_\sigma)^{m-s} \leq \prod_{j=1}^m x^\ell_{\sigma\setminus\sigma_j}.
\]
Since the sequence $x^\ell_\tau$ converges to $x^0_\tau$, for every $\tau\subseteq\sigma$, and the inequality above is determined by (continuous) polynomials, we thus get that 
\[
(x^0_\sigma)^{m-s} \leq \prod_{j=1}^m x^0_{\sigma\setminus\sigma_j}
\]
holds true, hence showing the Type-$1$ inequalities in \eqref{eq:Algebraic}. An analogous argument shows that $(x^0_\tau)_{\tau\subseteq\sigma}$ fulfills every inequality of Type-$2$ in \eqref{eq:Algebraic}. Thus, we obtain that $(x_\tau^0)_{\tau\subseteq\sigma}\in\mathcal C_\sigma$. Therefore, by Zorn's Lemma we obtain the existence of a minimal element $(t_\tau)_{\tau\subseteq\sigma}\in\mathcal C_\sigma$.

We finally prove $(iii)$. By the minimality of the tuple $(t_\tau)_{\tau\subseteq\sigma} \in \mathcal C_\sigma$, for each $i\in\sigma$ there exists an inequality of Type-$(1)$ or Type-$(2)$ involving $t_{\{i\}}$ in which equality holds. If for some $i\in\sigma$ there is equality in an inequality of Type-$(1)$ this means that there exists $\beta(i):=(\sigma_1^i,\ldots,\sigma_{m_i}^i)$ an irreducible $s_i$-cover of $\sigma$ with $\sigma_{j_0}^i=\sigma\setminus\{i\}$ for some $j_0\in\{1,\ldots,m_i\}$ such that
\[t_\sigma^{m_i-s_i}=\prod_{j=1}^{m_i}t_{\sigma\setminus\sigma^i_j}=t_{\{i\}}\prod_{j\in [m_i]\setminus\{j_0\}}t_{\sigma\setminus\sigma^i_j}.\]
If for some $i\in\sigma$ there is equality in an inequality of Type-$(2)$ this means that there exists a subset $\tau(i) $ of $\sigma$ such that $i\in\tau(i)$ and 
\begin{equation}\label{eq: equality type 2}
    t_{\tau(i)}=\prod_{j\in\tau(i)}t_{\{j\}}.
\end{equation}
Moreover, by the minimality of $(t_\tau)_{\tau\subseteq\sigma}$
there exists $(\sigma^i_1,\ldots,\sigma^i_{m_i})$ an irreducible $s_i$-cover of $\sigma$ such that $\sigma^i_{j_0}=\sigma\setminus\tau(i)$ and with equality in an inequality of Type-$(1)$. Otherwise we 
could replace $t_{\tau(i)}$ by a strictly smaller value
verifying all the conditions 
imposed in \eqref{eq:Algebraic}.
Therefore,
\begin{eqnarray}\label{eq: equality type 2 reduction}
    t_\sigma^{m_i-s_i}&=&\prod_{j=1}^{m_i}t_{\sigma\setminus\sigma^i_j}=t_{\tau(i)}\prod_{j\in [m_i]\setminus\{j_0\}}t_{\sigma\setminus\sigma^i_j}\nonumber\\
    &=&\left(\prod_{j\in\tau(i)}t_{\{j\}}\right)\left(\prod_{j\in [m_i]\setminus\{j_0\}}t_{\sigma\setminus\sigma^i_j}\right),
\end{eqnarray}
where in the last equality we have used \eqref{eq: equality type 2}. Now, we notice that taking the complements in $\sigma$ of each element of a $s$-cover of $\sigma$, with $m$ sets, transforms it into a $(m-s)$-cover. Thus, $(\sigma\setminus\sigma^i_1,\ldots,\sigma\setminus\sigma^i_{m_i})$ is a $(m_i-s_i)$-cover of $\sigma$. Besides, since $\sigma\setminus\sigma^i_{j_0}=\tau(i)$, we have that $(\sigma\setminus\sigma^i_j:j\in[m_i]\setminus\{j_0\})\cup(\{j\}:j\in\tau(i))$ is also a $(m_i-s_i)$-cover of $\sigma$ with $m_i-1+|\tau(i)|$ sets. Therefore, $\beta(i):=(\sigma^i_j:j\in[m_i]\setminus\{j_0\})\cup(\sigma\setminus\{j\}:j\in\tau(i))$ is a $(s_i+|\tau(i)|-1)$-cover of $\sigma$ with equality in the inequality of Type-$(1)$ by \eqref{eq: equality type 2 reduction}.

Now we define $\beta:=\cup_{i\in\sigma}\beta(i)$. It is  a $s$-cover of $\sigma$ with 
\[
s=\sum_{i\in\mathcal{M}_1}s_i+\sum_{i\in\mathcal{M}_2}(s_i+|\tau(i)|-1),
\]
where $\mathcal{M}_1:=\{i\in\sigma: \text{ there exists }\,(\sigma^i_1,\ldots,\sigma^i_{m_i}) \text{ an irreducible }s_i\text{-cover of }\sigma \text{ with } \sigma^i_{j_0}=\sigma\setminus\{i\} \text{ for some } j_0\in[m_i]\text{ and equality in a Type-}(1) \text{ inequality}\}$
and $\mathcal{M}_2:=\sigma\setminus \mathcal{M}_1.$ The number of elements of $\beta$ is
\[
m=\sum_{i\in\mathcal{M}_1}m_i+\sum_{i\in\mathcal{M}_2}(m_i-1+|\tau(i)|).
\]
Then, we have equality in the inequality of Type-$(1)$ applied to the $s$-cover $\beta$ since
\begin{eqnarray*}
    t_\sigma^{m-s}&=&t_\sigma^{\sum_{i\in\mathcal{M}_1}(m_i-s_i)}t_\sigma^{\sum_{i\in\mathcal{M}_2}((m_i-1+|\tau(i)|)-(s_i+|\tau(i)|-1))}\nonumber\\
    &=&\left[\prod_{i\in\mathcal{M}_1}t_\sigma^{m_i-s_i}\right]\left[\prod_{i\in\mathcal{M}_2}t_\sigma^{m_i-s_i}\right]\nonumber\\
    &=&\left[\prod_{i\in\mathcal{M}_1}\left(\prod_{\tau\in\beta(i)}t_{\sigma\setminus\tau}\right)\right]\left[\prod_{i\in\mathcal{M}_2}\left(\prod_{\tau\in\beta(i)}t_{\sigma\setminus\tau}\right)\right]\nonumber\\
    &=&\prod_{\tau\in\beta}t_{\sigma\setminus\tau}.
\end{eqnarray*}
Considering $\beta'=\beta\setminus(\sigma\setminus\{i\}:i\in\sigma)$, which is a $(s-|\sigma|+1)$-cover of $\sigma$ with $m-|\sigma|$ elements, we can rewrite the equality above as
\begin{equation}\label{eq: eq equality beta}
    t_\sigma^{m-s}=\left(\prod_{\tau\in\beta'}t_{\sigma\setminus\tau}\right)\left(\prod_{i\in\sigma}t_{\{i\}}\right).
\end{equation}
Besides, applying inequality of Type-$(1)$ to the cover $\beta'$ (it can be also applied to reducible covers with the same steps done in \eqref{eq: eq equality beta} but with inequality instead of equality) 
we get that
\begin{equation}\label{eq: ineq beta'}
  t_\sigma^{m-s-1}\leq\prod_{\tau\in\beta'}t_{\sigma\setminus\tau}. 
\end{equation}
Therefore, from \eqref{eq: eq equality beta} and \eqref{eq: ineq beta'} it follows that
\[
t_\sigma\geq\prod_{i\in\sigma}t_{\{i\}}.
\]
On the other hand, taking $\tau=\sigma$ in the inequality of Type-$(2)$ we have the reverse inequality, $t_\sigma\leq\prod_{i\in\sigma}t_{\{i\}}$. Therefore, 
\[
t_\sigma=\prod_{i\in\sigma}t_{\{i\}}.
\]

Let now $\tau\subset\sigma$. Since $\tau\cup(\{i\}:i\in\sigma\setminus\tau)$ is a $1$-cover of $\sigma$ with $|\sigma|-|\tau|+1$ sets, we have that $(\sigma\setminus\tau)\cup(\sigma\setminus\{i\}:i\in\sigma\setminus\tau)$ is a $(|\sigma|-|\tau|)$-cover of $\sigma$. Then,
\[
t_\sigma\leq t_\tau\left(\prod_{i\in\sigma\setminus\tau}t_{\{i\}}\right)\leq\left(\prod_{i\in\tau}t_{\{i\}}\right)\left(\prod_{i\in\sigma\setminus\tau}t_{\{i\}}\right)=\prod_{i\in\sigma}t_{\{i\}}=t_\sigma,
\]
where we have used the inequalities of Type-$(1)$ and Type-$(2)$ in the first and the second inequalities above, respectively.
Finally, we deduce that \[
t_\tau=\prod_{i\in\tau}t_{\{i\}}.
\]
\end{proof}
\section{Integral formula for $P_{H_\sigma}f$}\label{Sec: integral formula for PHf}

Recall that  $\Sigma=\{x\in\mathbb{R}^n:F(x)=0,\,\mathrm{with }\,\,F:\mathbb{R}^n\to\mathbb R\}$ is a hypersurface on $\mathbb R^n$ of class $\mathcal C^2$ if $F\in\mathcal{C}^2(\mathbb R^n)$ and $\nabla F(x)\neq 0$, for every $x\in\Sigma$. The Gauss map of $\Sigma$ is the map $N:\Sigma\to\mathbb S^{n-1}$ that for every $x\in\Sigma$ gives its unitary normal vector and it is given by 
\[
N(x)=\frac{\nabla F(x)}{\Vert \nabla F(x)\Vert_2}.
\]
Notice that orientation induced by $N$ depends on the choice of the defining function $F$. In particular, if    \[
\Sigma=\mathrm{bd}(K) \quad \mathrm{ with }\quad K=\{x\in\mathbb R^n:F(x)\geq 0\}\in\mathcal{ K}^n,
\] then $N$ points towards the interior of $K$. Thus, in this setting, $N$ is the inward-pointing unit normal. Thus, from now on we adopt this convention.
The differential function of $N$ denoted by  $dN_x:T_x\Sigma\to T_{N(x)}\mathbb S^{n-1}\equiv T_x\Sigma$ is defined over tangent space of $\Sigma$  in $x$,
$$T_x\Sigma:=\left\{v\in\mathbb R^n:\left\langle v,\frac{\nabla F(x)}{\Vert \nabla F(x)\Vert_2}\right\rangle=0\right\}=\nabla F(x)^\bot.$$
The eigenvalues of $-dN_x$, in $x\in\Sigma$, will be denoted by $\{\kappa_1(x),\ldots,\kappa_{n-1}(x)\}$ and they are known to be the principal curvatures of $\Sigma$ in $x$. Besides, these eigenvalues have special interest since the Gauss-Kronecker curvature of $\Sigma$ in $x\in\Sigma$ is given by the product of the principal curvatures in $x$,
\[
\kappa_\Sigma(x):=\prod_{i=1}^{n-1}\kappa_i(x).
\]
 Besides, if the hypersurface $\Sigma=\mathrm{bd}(K)$ is orientable  with $K\in\mathcal K^n$, we have that $\kappa_\Sigma\geq 0$ with all the principal curvatures non-negative. Moreover, if $\kappa_\Sigma>0$ for $\Sigma=\mathrm{bd}(K)$ with $K\in\mathcal K^n$, then $K$ is strictly convex \cite{Go05, Sc14}. The Gauss-Kronecker curvature can also be computed as
 \begin{equation}\label{eq: Gauss curvature and -dNx}
    \kappa_\Sigma(x)=\mathrm{det}(-dN_x)=\frac{\mathrm{det}(-D^2F(x)|_{\nabla F(x)^\bot})}{\Vert \nabla F(x)\Vert_2^{n-1}},
 \end{equation}
 where $D^2F(x)|_{\nabla F(x)^\bot}$ is the Hessian matrix of $F$ restricted to $\nabla F(x)^\bot$, \cite{DC16, Go05}.

 The following result is classical. In fact, for $n=3$ it goes back to Gauss' notion of \emph{curvatura integra} as the area of the spherical image \cite{Ga1828}, together with Hadamard's theorem that a closed surface with positive Gaussian curvature bounds a convex body and has a bijective Gauss map \cite{Ha1897}. In arbitrary dimension it is the convex case of Hopf's theorem \cite{Ho1926}: for every closed hypersurface of $\mathbb{R}^n$, the total Gauss--Kronecker curvature equals $\deg(N)\,S(\mathbb{B}^n)$, and $\deg(N)=1$ when $\Sigma$ bounds a convex body (see also \cite{CL1957}).
\begin{lemma}\label{lem: integral curvature}
    Let $\Sigma\subset\mathbb{R}^n$ be a compact and  orientable hypersurface of class $\mathcal{C}^2$ such that $\Sigma=bd (K)$ with $K\in\mathcal{K}^n$ and $\kappa_\Sigma>0$. Then,
    \[\int_\Sigma \kappa_{\Sigma}(x)d\mathcal{H}^{n-1}(x)=S(\mathbb{B}^{n}),\]
    where $d\mathcal{H}^{n-1}$ is the $(n-1)$-dimensional Hausdorff measure.
\end{lemma}
Recall that, for any $A\subseteq\mathbb R^n$, the $(n-1)$-dimensional Hausdorff measure is defined as
\[
\mathcal{H}^{n-1}(A)
=
\lim_{\delta\to0^+}
\mathcal{H}^{n-1}_{\delta}(A),
\]
where
\[
\mathcal{H}^{n-1}_{\delta}(A)
=
\frac{\mathrm{vol}_{n-1}\left(\mathbb B^{n-1}\right)}{2^{n-1}}
\inf\left\{
\sum_{j=1}^{\infty}
\bigl(\operatorname{diam} U_j\bigr)^{n-1}
:
A\subseteq\displaystyle\bigcup_{j=1}^{\infty}U_j,
\operatorname{diam} U_j\leq\delta
\right\},
\]
for $\delta>0$.


For the proof of Lemma \ref{lem: integral curvature}, observe that, since $\kappa_\Sigma=\det(-dN_x)>0$, the body $K$ is strictly convex and the Gauss map $N:\Sigma\to\mathbb{S}^{n-1}$ is a global $\mathcal{C}^1$ diffeomorphism. Hence, by the theorem of change of variables applied to $N$,
\[\int_{\Sigma}\kappa_\Sigma(x)d\mathcal{H}^{n-1}(x)=\int_\Sigma |\mathrm{det}(-dN_x)|d\mathcal{H}^{n-1}(x)=\int_{\mathbb S^{n-1}}d\mathcal{H}^{n-1}(x)=S(B^{n}).\]
In other words, the total Gauss--Kronecker curvature of $\Sigma$ equals the $(n-1)$-dimensional measure of its spherical image, namely $S(\mathbb{B}^n)$.

The proof of \eqref{eq: diferential formula integral projection} makes use of Federer's Coarea Formula \cite{F014},
\begin{equation}\label{eq: coarea formula}
    \int_{\mathbb{R}^n}\varphi(x)\Vert \nabla F(x) \Vert_2 dx=\int_{\mathbb{R}}\int_{\{x\in \mathbb{R}^n:F(x)=t\}}\varphi(x)d\mathcal{H}^{n-1}(x)dt,
\end{equation}
where the function $F:\mathbb{R}^n\to\mathbb{R}$ is locally Lipschitz and $\varphi:\mathbb{R}^n\to[0,\infty)$ is a measurable function. Then, we show the new integral formula for the projection of a log-concave function onto a linear coordinate subspace.
\begin{lemma}\label{lemma:intprojection}
Let $\sigma\subset[n]$. Let $f\in\mathcal{C}^2(\mathbb{R}^n)\cap \mathcal{Q}K(\mathbb R^n)$ be an integrable function such that $\Vert f\Vert_\infty<\infty$ and, for almost every $x\in \mathrm{supp}(f)$, $\nabla_{H_\sigma^\bot} f(x)\neq 0$. Then,
 \[
    \int_{H_\sigma}P_{H_\sigma}f(y)dy=\frac{1}{S(\mathbb B^{n-|\sigma|})}\scalebox{1.5}{$\displaystyle\int$}_{\mathbb R^n}\frac{\mathrm{det}\left(\left.-D^2_{H_\sigma^\bot}f(x)\right|_{\nabla_{H_\sigma^\bot} f(x)^\bot\cap H_\sigma^\bot}\right)}{\Vert \nabla_{H_\sigma^\bot} f(x)\Vert_2^{n-|\sigma|-2}}dx.
    \]
   \end{lemma}
\begin{proof}
    Let $y\in H_\sigma$ be fixed and $g_y: H_\sigma^\bot\to[0,\infty)$ be the function given by $g_y(x)=f(x+y)$ for every $x\in H_\sigma^\bot$. Notice also that $g_y\in \mathcal{Q}K(H_\sigma^\bot) $ is an integrable function of class $\mathcal{C}^2$ in $H_\sigma^\bot$ since $f\in\mathcal{C}^2(\mathbb{R}^n)\cap \mathcal{Q}K(\mathbb R^n)$. Let us consider the level sets 
    \[
    L_{g_y,t}:=\{x\in H_\sigma^\bot:\, g_y(x)\geq t\}
    \]
    which are convex and 
    \[
\Sigma_{g_y,t}:=\mathrm{relbd}\left(L_{g_y,t}\right)=\{x\in H_\sigma^\bot:\, g_y(x)= t\}
    \]
    which is a compact orientable hypersurface of class $\mathcal C^2$ in $H_\sigma^\bot$ with $\kappa_{\Sigma_{g_y,t}}>0$, for almost every $t\in(0,\sup\{g_y(x):x\in H_\sigma^\bot\})$. Therefore, by Lemma \ref{lem: integral curvature} applied to $\Sigma_{g_y,t}$ it follows that
    \begin{equation}\label{eq: superficie de la bola}
        \int_{\Sigma_{g_y,t}}\kappa_{\Sigma_{g_y,t}}(x)d\mathcal{H}^{n-|\sigma|-1}(x)=S(\mathbb B^{n-|\sigma|}).
    \end{equation}
    Now, we apply the coarea formula  \eqref{eq: coarea formula} to 
$\varphi(x)=\kappa_{\Sigma_{g_y,g_y(x)}}(x)$ which is a strictly positive function, and $F=g_y$, which is locally Lipschitz since $g_y\in\mathcal C^2$ and $\sup_{x\in H_\sigma^\bot}\{g_y(x)\}<\infty$. Then, by \eqref{eq: coarea formula} and \eqref{eq: superficie de la bola} together with $P_{H_\sigma}f(y)=\sup\{g_y(x):x\in H_\sigma^\bot\}$ we get that
    \begin{equation}\label{eq: proyeccion de f como función}
    \begin{split}
        \int_{H_\sigma^\bot}\kappa_{\Sigma_{g_y,g_y(x)}}(x) & \Vert \nabla g_y(x)\Vert_2dx \\
        &=\int_0^{\sup\{g_y(x):\,x\in H_\sigma^\bot\}}\int_{\Sigma_{g_y,t}}\kappa_{\Sigma_{g_y,t}}(x)d\mathcal{H}^{n-|\sigma|-1}(x)dt\\
        &=S(\mathbb B^{n-|\sigma|})\int_0^{\sup\{g_y(x):\,x\in H_\sigma^\bot\}}dt\\
        &=S(\mathbb B^{n-|\sigma|})P_{H_\sigma}f(y).
\end{split}
    \end{equation}
    
    Finally, integrating over $H_\sigma$ both sides of \eqref{eq: proyeccion de f como función} and using \eqref{eq: Gauss curvature and -dNx}  we conclude the proof,
    \begin{eqnarray*}
        \int_{H_\sigma}P_{H_\sigma}f(y)dy&=&\frac{1}{S(\mathbb B^{n-|\sigma|})}\int_{H_\sigma}\int_{H_\sigma^\bot}\kappa_{\Sigma_{g_y,g_y(x)}}(x)\Vert \nabla g_y(x)\Vert_2dxdy\\
        &=&\frac{1}{S(\mathbb B^{n-|\sigma|})}\int_{H_\sigma}\int_{H_\sigma^\bot}\frac{\mathrm{det}\left(-\left.D^2g_y(x)\right|_{\nabla g_y(x)^\bot\cap H_\sigma^\bot}\right)}{\Vert \nabla g_y(x)\Vert_2^{n-|\sigma|-2}}dxdy.
    \end{eqnarray*}
    \end{proof}

\section{Sobolev type inequalities}\label{sec:SobolevInequs}

We start this section by recalling some well-known facts about the Minkowski functional (see \cite{AGM15, Sc14}).

\begin{lemma}\label{lem: propiedades del funcional de Minkowski}
    Let $0\in K\in\mathcal{K}^n$ and $t>0$. Then,
    \begin{itemize}
        \item[i)] $\{x\in\mathbb{R}^n:\,\Vert x\Vert_K\leq 1\}=K$,
        \item[ii)] $\Vert x\Vert_{tK}=\frac{1}{t}\Vert x\Vert_{K}.$
        \end{itemize}
        Moreover, if $0\in K_1\in\mathcal K^i$ and  $0\in K_2\in\mathcal K^{n-i}$, then
        \begin{itemize}
        \item [iii)] for every $(x,y)\in\mathbb R^i\times\mathbb R^{n-i}$ it holds that
        \[\Vert (x,y)\Vert_{\mathrm{conv}(\{K_1,K_2\})}=\Vert x\Vert_{K_1}+\Vert y\Vert_{K_2}.\]
    \end{itemize}
\end{lemma}

The following identity is well known and belongs to the folklore of the subject (see, e.g., \cite{AGM15} and \cite{Sc14}); we include its short proof for the sake of completeness.
\begin{lemma}\label{lem: calculo del volumen con el funcional de minkowski}
    Let $0\in K\in\mathcal{K}^n$. Then,
    \[n!\mathrm{vol}_n(K)=\int_{\mathbb{R}^n}e^{-\Vert x\Vert_K}dx.\]
\end{lemma}
\begin{proof}
By Fubini's Theorem we have that
\[
 \int_{\mathbb{R}^n}e^{-\Vert x\Vert_K}dx=\int_{\mathbb{R}^n}\int_{\Vert x\Vert_K}^\infty e^{-t}dtdx=\int_0^\infty e^{-t}\int_{\{x\in\mathbb{R}^n:\Vert x\Vert_K\leq t\}}dxdt.
\]
Besides, by Lemma \ref{lem: propiedades del funcional de Minkowski}
\[
\{x\in\mathbb{R}^n:\Vert x\Vert_K\leq t\}=\{x\in\mathbb{R}^n:\Vert x\Vert_{tK}\leq 1\}=tK.
\]
Therefore,
\[
\int_0^\infty e^{-t}\int_{\{x\in\mathbb{R}^n:\Vert x\Vert_K\leq t\}}dxdt=\mathrm{vol}_n(K)\int_0^\infty e^{-t}t^ndt=n!\mathrm{vol}_n(K),
\]
proving the result.
\end{proof}

From now on we see how to exploit Theorem \ref{thm:AlgebraicTool} in order to show some Sobolev/ isoperimetric type results. We now prove Theorem \ref{thm:SobolevBTineq}.
\begin{thm}\label{thm: iso proyec}
     Let $f\in\mathcal{F}(\mathbb{R}^n)$. Then, there exists a coordinate box $B_f$ such that 
    \[
    \int_{\mathbb{R}^n}\frac{f(x)}{\Vert f\Vert_\infty}dx=\int_{\mathbb{R}^n}e^{-\Vert x\Vert_{B_f}}dx
    \]
    and
    \[
    \int_{H_\tau}\frac{P_{H_\tau}f(x)}{\Vert f\Vert_\infty}dx\geq\int_{H_\tau}P_{H_\tau}e^{-\Vert x\Vert_{B_f}}dx
    \]
    for any $\tau\subseteq[n].$
\end{thm}
\begin{proof}
    Let  $f\in\mathcal{F}(\mathbb{R}^n)$ and $(\sigma_1,\ldots,\sigma_m)$ a $s$-cover of $[n]$. Then, by \eqref{eq: functional ineq proyecciones BT} it follows that 
    \begin{equation*}
    \Vert f\Vert_\infty^{m-s}\left(\int_{\mathbb{R}^n}f(x)dx\right)^s\leq\frac{(n!)^s}{\prod_{j=1}^m|\sigma_j|!}\prod_{j=1}^m\int_{H_{\sigma_j}}P_{H_{\sigma_j}}f(x)dx.
\end{equation*}
Dividing both sides by $\Vert f\Vert_\infty^{m}$, which is not equal to zero since $f$ is not zero, we can rewrite it as
 \begin{equation}\label{eq: eq 1 proyec funtional reescrita}
    \left(\frac{1}{n!}\int_{\mathbb{R}^n}\frac{f}{\Vert f\Vert_\infty}(x)dx\right)^s\leq\prod_{j=1}^m\frac{1}{|\sigma_j|!}\int_{H_{\sigma_j}}P_{H_{\sigma_j}}\frac{f}{\Vert f\Vert_\infty}(x)dx.
\end{equation}
We can also apply \eqref{eq: eq 1 proyec funtional reescrita} to the function $P_{H_\tau}f$ for any $\tau\subseteq[n]$ and we get that
\begin{equation*}
    \left(\frac{1}{|\tau|!}\int_{H_\tau}\frac{P_{H_\tau}f}{\Vert P_{H_\tau}f\Vert_\infty}(x)dx\right)^s\leq\prod_{j=1}^m\frac{1}{|\tau_j|!}\int_{H_{\tau_j}}P_{H_{\tau_j}}\frac{P_{H_\tau}f}{\Vert P_{H_\tau}f\Vert_\infty}(x)dx,
\end{equation*}
with $(\tau_1,\ldots,\tau_m)$ a $s$-cover of $\tau$. It is easy to check that $\Vert f\Vert_\infty=\Vert P_{H_\tau}f\Vert_\infty$ and $P_{H_{\tau_j}}(P_{H_\tau}f)=P_{H_{\tau_j}}f$, for every $j=1,\ldots,m$. Therefore,
\begin{equation*}
    \left(\frac{1}{|\tau|!}\int_{H_\tau}P_{H_\tau}\left(\frac {f}{\Vert f\Vert_\infty}\right)(x)dx\right)^s\leq\prod_{j=1}^m\frac{1}{|\tau_j|!}\int_{H_{\tau_j}}P_{H_{\tau_j}}\left(\frac{f}{\Vert f\Vert_\infty}\right)(x)dx.
\end{equation*}
In particular, we can take the $1$-cover of $\tau$ given by $(\{j\}:j\in\tau)$ so that
\begin{equation}\label{eq: eq 1 proyec funtional tau }
    \frac{1}{|\tau|!}\int_{H_\tau}P_{H_\tau}\left(\frac{f}{\Vert f\Vert_\infty}\right)(x)dx\leq\prod_{j\in\tau}\int_{H_{\{j\}}}P_{H_{\{j\}}}\left(\frac{f}{\Vert f\Vert_\infty}\right)(x)dx,
\end{equation}

Now, we consider the set of positive real numbers ($f$ is a non-negative integrable function) given by \begin{equation}\label{set: BT que hace no vacio}
        \left\{C_\tau:=\frac{1}{|\tau|!}\int_{H_{\tau}}P_{H_{\tau}}\left(\frac{f}{\Vert f\Vert_\infty}\right)(x)dx: \,\tau\subseteq[n]\right\}.
    \end{equation}
     Then,  let $\mathcal{C}$ be the set  defined as in Theorem \ref{thm:AlgebraicTool} for the constants $C_\tau$ as given in \eqref{set: BT que hace no vacio}, for any $\tau\subseteq [n]$. By inequalities \eqref{eq: eq 1 proyec funtional reescrita} and \eqref{eq: eq 1 proyec funtional tau } the set \eqref{set: BT que hace no vacio} belongs to $\mathcal{C}\neq\emptyset$. Therefore, by Theorem \ref{thm:AlgebraicTool} there exists a set of positive real numbers $\{t_\tau\in(0,\infty):\tau\subseteq[n]\}\in\mathcal{C}$ such that
     \begin{equation}\label{eq:t_tau_prod}
    t_\tau=\prod_{j\in\tau}t_{\{j\}}
    \end{equation}
    for any $\tau\subseteq[n].$

    Let us consider now the coordinate box given by 
    \[B=\sum_{i=1}^n\frac{t_{\{i\}}}{2}\left[-e_i,e_i\right],\]
    whose volume is equal to
    \[\mathrm{vol}(B)=\prod_{i=1}^nt_{\{i\}}.\]
    By Lemma \ref{lem: calculo del volumen con el funcional de minkowski}, \eqref{eq:t_tau_prod} and the fact that $\{t_\tau\in(0,\infty):\tau\subseteq[n]\}\in\mathcal{C}$ it follows that
    \begin{eqnarray*}
        \int_{\mathbb{R}^n}e^{-\Vert x\Vert_B}dx&=&n!\mathrm{vol}(B)=n!\prod_{i=1}^nt_{\{i\}}=n!t_{[n]}\\
        &=&        \int_{\mathbb{R}^n}\frac{f}{\Vert f\Vert_\infty}(x)dx.
    \end{eqnarray*}
 Using the same arguments as above we can analogously prove that
\begin{eqnarray*}
     \int_{H_\tau}P_{H_\tau}e^{-\Vert x\Vert_{B}}dx&=&\int_{H_\tau}e^{-\Vert x\Vert_{P_{H_\tau}B}}dx=|\tau|!\mathrm{vol}(P_{H_\tau}B)\\
     &=&|\tau|!\prod_{i\in\tau}t_{\{i\}}=|\tau|!t_{\tau}\\
        &\leq&        \int_{H_\tau}P_{H_\tau}\left(\frac{f}{\Vert f\Vert_\infty}\right)(x)dx,
\end{eqnarray*}
for every $\tau\subseteq [n]$, concluding the proof.
\end{proof}

\begin{cor}
    Theorem \ref{thm:SobolevBTineq} follows from Lemma \ref{lemma:intprojection} and Theorem \ref{thm: iso proyec}.
\end{cor}

We now show Theorem \ref{thm:SobolevLiakopoulos}. Notice that we use below the analogous to Theorem \ref{thm:AlgebraicTool} but when replacing all instances of $\geq$ by $\leq$, minimal by maximal, and so on.
\begin{proof}[Proof of Theorem \ref{thm:SobolevLiakopoulos}]
     Let $f\in\mathcal{F}(\mathbb{R}^n)$ such that $0\in\mathrm{int}(\mathrm{supp }(f))$ and $(\sigma_1,\ldots,\sigma_m)$ a $s$-cover of $[n].$ Then by \eqref{eq: ABBC funcional secciones} we have that
    \begin{equation}\label{eq: David teorema funcional sec 2}
        \left(n^n\int_{\mathbb{R}^n}f^n(x)dx\right)^s\geq\prod_{j=1}^m|\sigma_j|^{|\sigma_j|}\int_{H_{\sigma_j}}f^{|\sigma_j|}(x)dx.
    \end{equation}
If we apply \eqref{eq: David teorema funcional sec 2} to the restriction of $f$ to the subspace $H_\tau$, which is also an integrable log-concave function on $H_\tau$ for any $\tau\subseteq[n]$, we get that 
\begin{equation*}
    \left(|\tau|^{|\tau|}\int_{H_\tau}f^{|\tau|}(x)dx\right)^s\geq\prod_{j=1}^m|\tau_j|^{|\tau_j|}\int_{H_{\tau_j}}f^{|\tau_j|}(x)dx,
\end{equation*}
where $(\tau_1,\ldots,\tau_m)$ is a $s$-cover of $\tau.$ In particular, we can consider the $1$-cover of $\tau$ given by $(\{j\}:j\in \tau)$ and it follows that
\begin{equation}\label{eq: david teorema funcional sec 2 tau}
    |\tau|^{|\tau|}\int_{H_\tau}f^{|\tau|}(x)dx\geq\prod_{j\in\tau}\int_{H_{\{j\}}}f(x)dx.
\end{equation}

Now, we consider the set of positive real numbers ($f$ is a non-negative integrable function) given by \begin{equation}\label{set: Lia que hace no vacio}
        \left\{C_\tau:=|\tau|^{|\tau|}\int_{H_{\tau}}f^{|\tau|}(x)dx: \,\tau\subseteq[n]\right\}.
    \end{equation}
    Then,  let $\mathcal{C}$ be the set  defined as in Theorem \ref{thm:AlgebraicTool} for the constants $C_\tau$ as given in \eqref{set: Lia que hace no vacio}, for any $\tau\subseteq [n]$. By inequalities \eqref{eq: David teorema funcional sec 2} and \eqref{eq: david teorema funcional sec 2 tau} the set \eqref{set: Lia que hace no vacio} belongs to $\mathcal{C}\neq\emptyset$. Therefore, by Theorem \ref{thm:AlgebraicTool} there exists a set of positive real numbers $\{t_\tau\in(0,\infty):\tau\subseteq[n]\}\in\mathcal{C}$ such that
    \[
    t_\tau=\prod_{j\in\tau}t_{\{j\}}
    \]
    for any $\tau\subseteq[n].$

    We consider the crosspolytope 
    \[C=\mathrm{conv}\left(\left\{\pm\frac{t_{\{i\}}}{2}e_i:\,i=1,\ldots n\right\}\right),\] which has volume equal to 
    \[\mathrm{vol}(C)=\frac{1}{n!}\prod_{i=1}^nt_{\{i\}},\]
    and the function $g=e^{-\Vert\cdot\Vert_C}$. Using ii) of Lemma \ref{lem: propiedades del funcional de Minkowski}, Lemma \ref{lem: calculo del volumen con el funcional de minkowski} and the fact that $\{t_\tau\in(0,\infty):\tau\subseteq[n]\}\in\mathcal{C}$, we get that 
    \begin{eqnarray*}
        \int_{\mathbb{R}^n}g^n(x)dx&=&\frac{1}{n^n}\int_{\mathbb{R}^n}e^{-\Vert x\Vert_C}dx=\frac{n!}{n^n}\mathrm{vol}(C)\\
        &=&\frac{1}{n^n}\prod_{i=1}^nt_{\{i\}}=\frac{1}{n^n}t_{[n]}\\
        &=&        \int_{\mathbb{R}^n}f^n(x)dx.
    \end{eqnarray*}
Notice that by the same arguments as above we get for every $\tau\subseteq [n]$ that
\begin{eqnarray*}
     \int_{H_\tau}g^{|\tau|}(x)dx&=&\frac{1}{|\tau|^{|\tau|}}\int_{H_\tau}e^{-\Vert x\Vert_C}dx=\frac{|\tau|!}{|\tau|^{|\tau|}}\mathrm{vol}(C\cap H_\tau)\\
     &=&\frac{1}{|\tau|^{|\tau|}}\prod_{i\in\tau}t_{\{i\}}=\frac{1}{|\tau|^{|\tau|}}t_{\tau}\\
        &\geq&        \int_{H_\tau}f^{|\tau|}(x)dx,
\end{eqnarray*}
concluding the proof.
\end{proof}

We now turn to the proof of Theorem \ref{thm:localSobolevBT}.

\begin{thm}\label{thm:LocalBTprojection}
    Let $f\in\mathcal{F}(\mathbb{R}^n)$ and $\sigma\subset[n]$. Then, there exists a certain convex body
    \[K_f=\mathrm{conv}\left(\left\{\sum_{i\in\sigma}c_i[-e_i,e_i],L
    \right\}\right),\]
     where $c_i>0$, $i\in\sigma$, $L$ is a $(n-|\sigma|)$-dimensional convex body contained in $H_\sigma^\bot$, with $0\in L$,
    such that   
    \[
    \int_{\mathbb{R}^n}f(x)dx=\int_{\mathbb{R}^n}e^{-\Vert x\Vert_{K_f}}dx,
    \]
    \[
    \int_{H_\sigma^\bot\oplus H_\tau}P_{H_\sigma^\bot\oplus H_\tau}f(x)dx\geq\int_{H_\sigma^\bot\oplus H_\tau}P_{H_\sigma^\bot\oplus H_\tau}e^{-\Vert x\Vert_{K_f}}dx
    \]
    for every $\tau\subseteq\sigma$ and
    \[
\int_{H_\sigma^\bot}P_{H_\sigma^\bot}f(x)dx=\int_{H_\sigma^\bot}P_{H_\sigma^\bot}e^{-\Vert x\Vert_{K_f}}dx.
    \]
\end{thm}
\begin{proof}
Let $f\in\mathcal{F}(\mathbb{R}^n)$ and $(\sigma_1,\ldots,\sigma_m)$ a $s$-cover of $\sigma\subset[n]$. Using \eqref{eq: funcional local BT} we have that
\begin{equation*}
\begin{split}
    \left(\int_{H_\sigma^\bot}P_{H_\sigma^\bot}f(x)dx\right)^{m-s} & \left(\int_{\mathbb{R}^n}f(x)dx\right)^s \\ & \leq\frac{|\sigma|!^s}{\prod_{j=1}^m|\sigma_j|!}\prod_{j=1}^m\int_{H_\sigma^\bot\oplus H_{\sigma_j}}P_{H_\sigma^\bot\oplus H_{\sigma_j}}f(x)dx.
\end{split}
\end{equation*}
Notice that when $\int f=0$, Theorem \ref{thm:LocalBTprojection} holds true. Thus, we can assume that $\int f>0$. Dividing both sides by 
\[\left(\int_{H_\sigma^\bot}P_{H_\sigma^\bot}f(x)dx\right)^{m},\]
which is not zero since the integral of $f$ is non-zero, and reordering it can be rewritten as
\begin{equation}\label{eq: eq 1 proyec local funtional reescrita}
    \left(\frac{1}{|\sigma|!}\frac{\int_{\mathbb{R}^n}f(x)dx}{\int_{H_\sigma^\bot}P_{H_\sigma^\bot}f(x)dx}\right)^s\leq\prod_{j=1}^m\frac{1}{|\sigma_j|!}\frac{\int_{H_\sigma^\bot\oplus H_{\sigma_j}}P_{H_\sigma^\bot\oplus H_{\sigma_j}}f(x)dx}{\int_{H_\sigma^\bot}P_{H_\sigma^\bot}f(x)dx}.
\end{equation}

Now, we consider the function $P_{H_\sigma^\bot\oplus H_\tau}f$, which is also an integrable log-concave function on $H_\sigma^\bot\oplus H_\tau$ for any $\tau\subseteq\sigma$, and we get that
\begin{equation*}
     \left(\frac{1}{|\tau|!}\frac{\int_{H_\sigma^\bot\oplus H_\tau}P_{H_\sigma^\bot\oplus H_\tau}f(x)dx}{\int_{H_\sigma^\bot}P_{H_\sigma^\bot}P_{H_\sigma^\bot\oplus H_\tau}f(x)dx}\right)^s\leq\prod_{j=1}^m\frac{1}{|\tau_j|!}\frac{\int_{H_\sigma^\bot\oplus H_{\tau_j}}P_{H_\sigma^\bot\oplus H_{\tau_j}}P_{H_\sigma^\bot\oplus H_\tau}f(x)dx}{\int_{H_\sigma^\bot}P_{H_\sigma^\bot}P_{H_\sigma^\bot\oplus H_\tau}f(x)dx},
\end{equation*}
where $(\tau_1,\ldots,\tau_m)$ is a $s$-cover of $\tau$. Taking into account that $P_{H_\sigma^\bot\oplus H_{\tau_j}}P_{H_\sigma^\bot\oplus H_\tau}f=P_{H_\sigma^\bot\oplus H_{\tau_j}}f$, for any $\tau_j\subseteq\tau$, it can be rewritten as
\begin{equation*}
    \left(\frac{1}{|\tau|!}\frac{\int_{H_\sigma^\bot\oplus H_\tau}P_{H_\sigma^\bot\oplus H_\tau}f(x)dx}{\int_{H_\sigma^\bot}P_{H_\sigma^\bot}f(x)dx}\right)^s\leq\prod_{j=1}^m\frac{1}{|\tau_j|!}\frac{\int_{H_\sigma^\bot\oplus H_{\tau_j}}P_{H_\sigma^\bot\oplus H_{\tau_j}}f(x)dx}{\int_{H_\sigma^\bot}P_{H_\sigma^\bot}f(x)dx}.
\end{equation*}
In particular, taking $(\{j\}:j\in \tau)$ a $1$-cover of $\tau$ it follows that
\begin{equation}\label{eq: eq 1 proyec local funtional tau}
    \frac{1}{|\tau|!}\frac{\int_{H_\sigma^\bot\oplus H_\tau}P_{H_\sigma^\bot\oplus H_\tau}f(x)dx}{\int_{H_\sigma^\bot}P_{H_\sigma^\bot}f(x)dx}\leq\prod_{j\in\tau}\frac{\int_{H_\sigma^\bot\oplus H_{\{j\}}}P_{H_\sigma^\bot\oplus H_{\{j\}}}f(x)dx}{\int_{H_\sigma^\bot}P_{H_\sigma^\bot}f(x)dx}.
\end{equation}

We also consider the set of positive real numbers given by 
\begin{equation}\label{set: BT local que hace no vacio}
        \left\{C_\tau:=\frac{1}{|\tau|!}\frac{\int_{H_\sigma^\bot\oplus H_{\tau}}P_{H_\sigma^\bot\oplus H_{\tau}}f(x)dx}{\int_{H_\sigma^\bot}P_{H_\sigma^\bot}f(x)dx}: \,\tau\subseteq\sigma\right\}.
    \end{equation}
     Then,  let $\mathcal{C}$ be the set  defined as in Theorem \ref{thm:AlgebraicTool} for the constants  $C_\tau$ as given in \eqref{set: BT local que hace no vacio}, for any $\tau\subseteq \sigma$. By inequalities \eqref{eq: eq 1 proyec local funtional reescrita} and \eqref{eq: eq 1 proyec local funtional tau} the set \eqref{set: BT local que hace no vacio} belongs to $\mathcal{C}\neq\emptyset$. Therefore, by Theorem \ref{thm:AlgebraicTool} there exists a set of positive real numbers $\{t_\tau\in(0,\infty):\tau\subseteq \sigma\}\in\mathcal{C}$ such that
     \begin{equation}\label{eq:localProdtTau}
    t_\tau=\prod_{j\in\tau}t_{\{j\}}
    \end{equation}
    for any $\tau\subseteq\sigma.$

    Let us consider the convex body given by 
    \[K=\mathrm{conv}\left(\left\{\sum_{i\in\sigma}\frac{t_{\{i\}}}{2}[-e_i,e_i],L
    \right\}\right)\]
    where $L$ is a $(n-|\sigma|)$-dimensional convex body within $\langle \{e_i:i\in[n]\setminus\sigma\}\rangle$ containing the origin with
    \begin{equation}\label{eq: condicion sobre L}
        \int_{H_\sigma^\bot}P_{H_\sigma^\bot}f(x)dx=(n-|\sigma|)!\mathrm{vol}(L).
    \end{equation}
   Moreover, by iii) in Lemma \ref{lem: propiedades del funcional de Minkowski} and Lemma \ref{lem: calculo del volumen con el funcional de minkowski} the volume of $K$ is equal to
    \begin{eqnarray*}
        \mathrm{vol}(K)&=&\frac{1}{n!}\int_{\mathbb{R}^n}e^{-\Vert x\Vert_K}dx\\
        &=&\frac{1}{n!}\int_{H_\sigma}\int_{H_\sigma^\bot}e^{-\left(\Vert y\Vert_{\sum_{i\in\sigma}\frac{t_{\{i\}}}{2}[-e_i,e_i]}+\Vert z\Vert_L\right)}dzdy\\
        &=&\frac{|\sigma|!(n-|\sigma|)!}{n!}\mathrm{vol}\left(\sum_{i\in\sigma}\frac{t_{\{i\}}}{2}[-e_i,e_i]\right)\mathrm{vol}(L)\\
        &=&\frac{1}{{n\choose n-|\sigma|}}\mathrm{vol}(L)\prod_{i\in\sigma}t_{\{i\}}.
    \end{eqnarray*}
Using \eqref{eq: condicion sobre L}, \eqref{eq:localProdtTau}, and the fact that $\{t_\tau:\tau\subseteq\sigma\}\in\mathcal C$ we get that 
\[
\begin{split}
    \int_{\mathbb{R}^n}e^{-\Vert x\Vert_{K}}dx & =(n-|\sigma|)!\mathrm{vol}(L)|\sigma|!\prod_{i\in\sigma}t_{\{i\}} \\
    & = \left(\int_{H_\sigma^\bot}P_{H_\sigma^\bot}f(x)dx\right)|\sigma|!t_\sigma \\ 
    & = \int_{H_\sigma^\bot}P_{H_\sigma^\bot}f(x)dx\frac{\int_{\mathbb{R}^n}f(x)dx}{\int_{H_\sigma^\bot}P_{H_\sigma^\bot}f(x)dx}\\
    & = \int_{\mathbb{R}^n}f(x)dx,
\end{split}
\]
which proves the first identity of this result.

Besides, for any $\tau\subseteq\sigma$, we firstly have that
\begin{equation}\label{eq: primera parte condicion 2 proyec local}
\begin{split}
     \int_{H_\sigma^\bot\oplus H_\tau}P_{H_\sigma^\bot\oplus H_\tau}e^{-\Vert x\Vert_{K}}dx & = \int_{H_\sigma^\bot\oplus H_\tau}e^{-\Vert x\Vert_{P_{H_\sigma^\bot\oplus H_\tau}K}}dx\\
     & =(n-|\sigma|+|\tau|)!\mathrm{vol}\left(P_{H_\sigma^\bot\oplus H_\tau}K\right)\\
     & =(n-|\sigma|)!|\tau|!\mathrm{vol}\left(P_{ H_\tau}\sum_{i\in\sigma}\frac{t_{\{i\}}}{2}[-e_i,e_i]\right)\mathrm{vol}(L)\\
     & =(n-|\sigma|)!\mathrm{vol}(L)|\tau|!\left(\prod_{i\in\tau}t_{\{i\}}\right),
\end{split}
\end{equation}
 where we have used similar ideas to what we did for the first identity and the fact that $P_He^{-\Vert\cdot\Vert_K}=e^{-\Vert\cdot\Vert_{P_HK}}$.
Finally, by Theorem \ref{thm:AlgebraicTool} and \eqref{eq: condicion sobre L}, \eqref{eq: primera parte condicion 2 proyec local} is equal to
\begin{eqnarray*}
   (n-|\sigma|)!\mathrm{vol}(L)|\tau|!t_\tau
    &\leq&\int_{H_\sigma^\bot}P_{H_\sigma^\bot}f(x)dx\frac{\int_{H_\sigma^\bot\oplus H_{\tau}}P_{H_\sigma^\bot\oplus H_{\tau}}f(x)dx}{\int_{H_\sigma^\bot}P_{H_\sigma^\bot}f(x)dx}\\
    &=&\int_{H_\sigma^\bot\oplus H_{\tau}}P_{H_\sigma^\bot\oplus H_{\tau}}f(x)dx,
\end{eqnarray*}
which proves the remaining inequalities of the theorem, concluding the proof of this result.
\end{proof}
\begin{cor}
    Theorem \ref{thm:localSobolevBT} follows from Lemma \ref{lemma:intprojection} and  Theorem \ref{thm:LocalBTprojection}.
\end{cor}

\section{Isoperimetric results for general product measures}\label{sec:ProdMeasure}

We start this section by proving Theorem \ref{thm:ProductMeasure}.

\begin{proof}[Proof of Theorem \ref{thm:ProductMeasure}]
Let $(\sigma_1,\ldots,\sigma_m)$ be a $s$-cover of $[n]$. Then, by \eqref{eq:BTprodMeas}
    \begin{equation}\label{eq: BT ineq product measures dem iso 1}
    \mu_1\otimes\cdots\otimes\mu_n(K)^s\leq\prod_{i=1}^m\underset{j\in\sigma_i}{\otimes}\mu_j\left(P_{H_{\sigma_i}}K\right).
\end{equation}
If we consider the convex body $P_{H_\tau}K$ for any $\tau\subseteq[n]$, we also have 
 \begin{equation*}
    \underset{j\in\tau}{\otimes}\mu_j(P_{H_\tau}K)^s\leq\prod_{i=1}^m\underset{j\in\tau_i}{\otimes}\mu_j\left(P_{H_{\tau_i}}K\right).
\end{equation*}
where $(\tau_1,\ldots,\tau_m)$ is a $s$-cover of $\tau.$  It can be also applied to the $1$-cover of $\tau$ given by $(\{j \}:\,j\in\tau)$ so that
\begin{equation}\label{eq: BT ineq product measures dem iso 1 tau}
     \underset{j\in\tau}{\otimes}\mu_j(P_{H_\tau}K)\leq\prod_{i\in\tau}\mu_i\left(P_{H_{\{i\}}}K\right).
\end{equation}

Let us now consider the set of positive real numbers
\begin{equation}\label{set: BT product measures que hace no vacio}
    \left\{C_{\tau}:=\underset{j\in\tau}{\otimes}\mu_j\left(P_{H_{\tau}}K\right):\, \tau\subseteq[n]\right\}.
\end{equation}
Then, let $\mathcal{C}$ be the set defined as in Theorem \ref{thm:AlgebraicTool} for the constants  $C_\tau$ as given in \eqref{set: BT product measures que hace no vacio}, for any $\tau\subseteq [n]$. By inequalities \eqref{eq: BT ineq product measures dem iso 1} and \eqref{eq: BT ineq product measures dem iso 1 tau} the set \eqref{set: BT product measures que hace no vacio} belongs to $\mathcal{C}\neq\emptyset$. Therefore, by Theorem \ref{thm:AlgebraicTool} there exists a set of positive real numbers $\{t_\tau\in(0,\infty):\tau\subseteq[n]\}\in\mathcal{C}$ such that
     \[
    t_\tau=\prod_{j\in\tau}t_{\{j\}}
    \]
    for any $\tau\subseteq[n].$

Let $G_j:\left[\mathrm{supp}(\mu_j)\cup (-\mathrm{supp}(\mu_j))\right]\cap[0,+\infty]\to\mathbb{R}$ be the function given by 
\[
G_j(x)=\mu_j([-x,x])=\int_{-x}^xd\mu_j(y)
\]
for every $j=1,\ldots,n$. Each function $G_j$ is non-decreasing by the monotonicity of the measure $\mu_j$. Moreover, it is also strictly increasing since it is defined over $[\mathrm{supp}(\mu_j)\cup(-\mathrm{supp}(\mu_j))]\cap[0,\infty)$. Therefore, it is an injective function and its inverse is well-defined with
\[
G_j(G_j^{-1}(y))=y
\]    
for every $y\in\mathrm{Im}(G_j)$. Notice that due to the continuity of $\mu_j$, we thus have that $\mathrm{Im}(G_j)=[0,\sup G_j]$ is a segment (where for instance, $\sup G_j=+\infty$ in the case of the Lebesgue measure or $\sup G_j=1$ in case of a probability measure).

Observe that by Theorem \ref{thm:AlgebraicTool} 
\[
t_{\{j\}}\leq\mu_{j}\left(P_{H_{\{j\}}}K\right)\leq \sup G_j,
\]
and hence $t_{\{j\}}\in\mathrm{Im}(G_j)$ and 
\[
G_j(G_j^{-1}(t_{\{j\}}))=t_{\{j\}}
\]
Thus, let $B$ be the coordinate box given by
\[
B:=\prod_{j=1}^n\left[-G_j^{-1}\left(t_{\{j\}}\right),G_j^{-1}\left(t_{\{j\}}\right)\right].
\]

Then, it follows that
\begin{eqnarray*}
    \mu_1\otimes\cdots\otimes\mu_n(B)&=&\prod_{j=1}^n\mu_j\left(\left[-G_j^{-1}\left(t_{\{j\}}\right),G_j^{-1}\left(t_{\{j\}}\right)\right]\right)\\
    &=&\prod_{j=1}^nG_j\left(G_j^{-1}\left(t_{\{j\}}\right)\right)\\
    &=&\prod_{j=1}^nt_{\{j\}}\\
    &=&t_{[n]}\\
    &=&\mu_1\otimes\cdots\otimes\mu_n(K).
\end{eqnarray*}
Analogously for any $\tau\subseteq[n]$ we also have
\begin{eqnarray*}
    \underset{j\in\tau}{\otimes}\mu_j(P_{H_\tau}B)&=&\prod_{j\in\tau}\mu_j\left(\left[-G_j^{-1}\left(t_{\{j\}}\right),G_j^{-1}\left(t_{\{j\}}\right)\right]\right)\\
    &=&\prod_{j\in\tau}G_j\left(G_j^{-1}\left(t_{\{j\}}\right)\right)\\
    &=&\prod_{j\in\tau}t_{\{j\}}\\
    &=&t_{\tau}\\
    &\leq&\underset{j\in\tau}{\otimes}\mu_j(P_{H_\tau}K),\\
\end{eqnarray*}
which concludes the proof.
\end{proof}

Theorem \ref{thm:ProductMeasure} can be applied to several non-homogeneous product measures. For instance, let $\gamma_n$ be the Gaussian measure defined by
\[d\gamma_n(x):=\frac{1}{(2\pi)^\frac{n}{2}}e^{\frac{-\Vert x\Vert_2^2}{2}}dx.\]
Notice that if $B=\prod_{i=1}^n[a_i,b_i]$ is a coordinate box, for some $a_i<b_i$, $i=1,\dots,n$ then
\begin{eqnarray*}
    \gamma_n(B)&=&\int_B\frac{1}{(2\pi)^\frac{n}{2}}e^{\frac{-\Vert x\Vert_2^2}{2}}dx\\
    &=&\prod_{i=1}^n\int_{a_i}^{b_i}\frac{1}{\sqrt{2\pi}}e^{\frac{-x_i^2}{2}}dx_i\\
    &=&\prod_{i=1}^n\gamma_1([a_i,b_i]),
\end{eqnarray*}
i.e.~$\gamma_n$ is a product measure, and thus one can derive from Theorem \ref{thm:ProductMeasure} the following result.
\begin{cor}
    Let $K\in\mathcal K^n$. Then there exists a coordinate box $B_K$ such that
    \[\gamma_n(K)=\gamma_n(B_K)\]
    and 
    \[\gamma_{|\tau|}(P_{H_\tau}K)\geq\gamma_{|\tau|}(P_{H_\tau}B_K)\]
    for any $\tau\subseteq[n].$    
\end{cor}

\section{Further remarks}\label{sec:furtherRemarks}

One may observe that the optimizing functions in Theorems \ref{thm:SobolevBTineq}, \ref{thm:SobolevLiakopoulos}, and \ref{thm:localSobolevBT} are functions of the form $e^{-\|\cdot\|_K}$, for a certain convex body $K$ containing the origin. As mentioned in the introduction, these functions naturally induce an embedding of the space of convex bodies onto that of log-concave functions. However, we find another natural embedding, once we consider the family of characteristic functions 
\[
\chi_K(x):=\left\{\begin{array}{lc}
    1 &  \text{ if }x\in K\\
    0 & \text{otherwise.}
\end{array}\right.
\]
This raises the following natural question.
\begin{problem}
    Find Sobolev inequalities for sections or projections of log-concave functions where the extremizers are characteristic functions of a certain convex body.
\end{problem}
Notice that Alonso-Gutiérrez, Bernués, Brazitikos and Carbery \cite[Thm.~4.1]{ABBC21} proved that for every function $f:\mathbb{R}^n\to\mathbb R$ with compact support and
continuous on its support, and every $(\sigma_1,\dots,\sigma_m)$ a $s$-cover of $[n]$, then
\begin{equation}\label{eq:BTcharacte}
    \left(\int_{\mathbb R^n}|f|^{\frac{m}{s}}(x)dx\right)^s\leq\prod_{j=1}^m\int_{H_{\sigma_j}}P_{H_{\sigma_j}}|f|(x)dx.
\end{equation}
Moreover, \eqref{eq:BTcharacte} recovers the classical Bollobás-Thomason inequality when $f=\chi_K$ for $K\in\mathcal{K}^n$.
Unfortunately, the technique developed in Theorem \ref{thm:AlgebraicTool} cannot be used to derive Sobolev inequalities with the set of inequalities above.
\begin{problem}
    Find Sobolev inequalities using the set of inequalities of \eqref{eq:BTcharacte}.
\end{problem}
Analogously, this would lead to finding other sets of functional inequalities in the line of \eqref{eq:BTcharacte} extending some Loomis-Whitney type inequalities tight for characteristic functions.

\textit{Acknowledgements}: We would like to thank Julián Haddad and María A. Hernández Cifre for their useful comments that helped us improving the paper.

\end{document}